\documentclass[11pt]{amsart}
\usepackage{amscd}
\usepackage{amsmath}
\usepackage{amsxtra}
\usepackage{amsfonts}
\usepackage{amssymb}

\newtheorem{theorem}{Theorem}[section]
\newtheorem{corollary}[theorem]{Corollary}
\newtheorem{lemma}[theorem]{Lemma}
\newtheorem{proposition}[theorem]{Proposition}

\theoremstyle{definition}
\newtheorem{definition}[theorem]{Definition}
\newtheorem{remark}[theorem]{Remark}

\newtheorem{example}[theorem]{Example}
\theoremstyle{remark}

\renewcommand{\theclaim}{\textup{\theclaim}}

\newtheorem*{acknowledgements}{Acknowledgements}

\numberwithin{equation}{section}

\def\openone

{\mathchoice

{\hbox{\upshape \small1\kern-3.3pt\normalsize1}}

{\hbox{\upshape \small1\kern-3.3pt\normalsize1}}

{\hbox{\upshape \tiny1\kern-2.3pt\SMALL1}}

{\hbox{\upshape \Tiny1\kern-2pt\tiny1}}}

\makeatletter

\newbox\ipbox

\newcommand{\ip}[2]{\left\langle #1\, , \,#2\right\rangle}
\newcommand{\diracb}[1]{\left\langle #1\mathrel{\mathchoice

{\setbox\ipbox=\hbox{$\displaystyle \left\langle\mathstrut
#1\right.$}

\vrule height\ht\ipbox width0.25pt depth\dp\ipbox}

{\setbox\ipbox=\hbox{$\textstyle \left\langle\mathstrut
#1\right.$}

\vrule height\ht\ipbox width0.25pt depth\dp\ipbox}

{\setbox\ipbox=\hbox{$\scriptstyle \left\langle\mathstrut
#1\right.$}

\vrule height\ht\ipbox width0.25pt depth\dp\ipbox}

{\setbox\ipbox=\hbox{$\scriptscriptstyle \left\langle\mathstrut
#1\right.$}

\vrule height\ht\ipbox width0.25pt depth\dp\ipbox}

}\right. }

\newcommand{\dirack}[1]{\left. \mathrel{\mathchoice

{\setbox\ipbox=\hbox{$\displaystyle \left.\mathstrut
#1\right\rangle$}

\vrule height\ht\ipbox width0.25pt depth\dp\ipbox}

{\setbox\ipbox=\hbox{$\textstyle \left.\mathstrut
#1\right\rangle$}

\vrule height\ht\ipbox width0.25pt depth\dp\ipbox}

{\setbox\ipbox=\hbox{$\scriptstyle \left.\mathstrut
#1\right\rangle$}

\vrule height\ht\ipbox width0.25pt depth\dp\ipbox}

{\setbox\ipbox=\hbox{$\scriptscriptstyle \left.\mathstrut
#1\right\rangle$}

\vrule height\ht\ipbox width0.25pt depth\dp\ipbox}

} #1\right\rangle}

\newcommand{\cj}[1]{\overline{#1}}

\newcommand{\bz}{\mathbb{Z}}
\newcommand{\M}{\mathcal{M}}
\newcommand{\B}{\mathcal{B}}

\newcommand{\bn}{\mathbb{N}}

\def\blfootnote{\xdef\@thefnmark{}\@footnotetext}

\input xy
\xyoption{all}
\usepackage{amssymb}

\newcommand{\End}{\operatorname*{End}}

\newcommand{\Span}{\overline{\operatorname*{span}}}

\def\C{\mathcal{C}}

\def\H{\mathcal{H}}

\def\-{^{-1}}
\def\B{\mathcal{B}}

\def\m{\mathfrak m}
\def\O{\mathcal{O}}
\def\K{\mathcal{K}}

\def\L{\mathcal L}

\begin{document}

\title[Cuntz algebras and Markov measures]{Monic representations of the Cuntz algebra and Markov measures}
\author{Dorin Ervin Dutkay}

\address{[Dorin Ervin Dutkay] University of Central Florida\\
	Department of Mathematics\\
	4000 Central Florida Blvd.\\
	P.O. Box 161364\\
	Orlando, FL 32816-1364\\
U.S.A.\\} \email{Dorin.Dutkay@ucf.edu}

%

\author{Palle E.T. Jorgensen}
\address{[Palle E.T. Jorgensen]University of Iowa\\
Department of Mathematics\\
14 MacLean Hall\\
Iowa City, IA 52242-1419\\}\email{palle-jorgensen@uiowa.edu}

\thanks{} 

\begin{abstract}
   We study representations of the Cuntz algebras $\O_N$. While, for fixed $N$,  the set of equivalence classes of representations of $\O_N$ is known not to have a Borel cross section, there are various subclasses of representations which can be classified. We study monic representations of $\O_N$, that have a cyclic vector for the canonical abelian subalgebra.  We show that $\O_N$ has a certain universal representation  which contains all positive monic representations. A large class of examples of monic representations is based on Markov measures. We classify them and as a consequence we obtain that different parameters yield mutually singular Markov measure, extending the classical result of Kakutani. The monic representations based on the Kakutani measures are exactly the ones that have a one-dimensional cyclic $S_i^*$-invariant space.    
\end{abstract}
\maketitle \tableofcontents

\section{Introduction}

 The Cuntz algebra $\O_N$ is  indexed by an integer $N > 1$, where $N$ is the number of generators. As a $C^*$-algebra (denoted $\O_N$), it is defined by its generators and relations (the Cuntz-relations), and $\O_N$ is known to be a simple, purely infinite $C^*$-algebra, \cite{Cu77}. Further its $K$-groups are known. But its irreducible representations are highly subtle. To appreciate the importance of the study of representations of $\O_N$, recall that to specify a representation of $\O_N$ amounts to identifying a system of isometries in a Hilbert space $\H$, with mutually orthogonal ranges, and adding up to $\H$. But such orthogonal splitting in Hilbert space may be continued iteratively, and as a result, one gets links between the study of $\O_N$-representation on the one hand, to such neighboring areas as symbolic dynamics and to filters used in signal processing, corresponding to a system of $N$ uncorrelated frequency bands.
 
Returning to the subtleties of the representations of $\O_N$, and their equivalence classes, it is known that, for fixed $N$, that the set of equivalence classes of irreducible representations  of $\O_N$,  does not admit a Borel cross section; i.e., the equivalence classes, under unitary equivalence, does not admit a parameterization in the measurable Borel category. (Intuitively, they defy classification). Nonetheless, special families of inequivalent representations have been found, and they have a multitude of applications, both to mathematical physics \cite{BrJo02}, to the study of wavelets \cite{DuJo08a, DuJo07a, Jor06, Jor01}, to harmonic analysis \cite{Str89, DuJo9a, DuJo07b}, to the study of fractals as iterated function systems \cite{DuJo06a, DuJo11a}; and to the study of $\End(B(\H))$ (= endomorphisms)  where $\H$ is a fixed Hilbert space.  Hence it is of interest to identify both discrete and continuous series of representations of $\O_N$, as they arise in such applications.

 From Definition \ref{def2.0}, it is evident that the problem of finding representations of $\O_N$, in a Hilbert space, and their properties, is a rather abstract one, and daunting. Unless the problem is first pared down and structured, there is little one can do in the way of finding and classify $\O_N$-representations. There is even a theorem of Glimm \cite{Gli60, Gli61} to the effect all representations do not admit a Borel labeling; more precisely the set of equivalence classes of representations of $\O_N$ do not have a Borel cross section. Nonetheless the representations of $\O_N$ have a host of applications (e.g., wavelets, fractals, signal processing, quantum measurement and information theory \cite{BrJo02, DJ06, Jor06, Nel69}.)

     A more realistic approach is instead to analyze specific families of representations of $\O_N$.  Our present approach is two-fold: (i) we build a measure space $(\K_N, \B,  \mu)$, where $\K_N$ is a compact Hausdorff space, $\B$ is the Borel-sigma algebra,  and  $\mu$  a probability measure on $(\K_N, \B)$. We take $\K_N$ to be the symbol space consisting of the set of infinite words in the alphabet $\bz_N$, where $\bz_N$ is the cyclic group of order $N$. Equivalently $\K = \K_N$ is the infinite Cartesian product with $\bz_N$ on each factor. For Hilbert space $\H$  we then take $L^2(\mu)$ ($= L^2(\K_N, \B, \mu)$).

       But to get representations of $\O_N$, (ii) we must then first identify a system of isometries $S_i$, $i \in \bz_N$, satisfying the Cuntz relations; see Definition \ref{def2.0}. And the interplay between endomorphisms of $\K_N$ on the one hand, and the associated measures $\mu$ on the other, places strong restrictions of the admissible measures which must first be understood. To this end we turn the question into a problem in symbolic dynamics:  we wish to realize the respective shifts in the symbol space  $\K_N$:  there is one shift to the left $\sigma$, and a system  of of $N$ endomorphisms  $\sigma _i$, shifting to the right. For fixed $i$, $\sigma _i$ is shifting a symbol string to the right, and filling in the letter $i$ at the first slot; see Definition \ref{defcantor}. Now to get a representation of $\O_N$ from this, the measure  $\mu$   which is used must have a number of delicate properties, for example each of the  $N + 1$  shifts applied to  $\mu$ must be relatively absolutely continuous with respect to  $\mu$ itself, i.e., quasi-invariance. Now shift to the left $\sigma$ is only an endomorphism in $\K_N$, and so far the quasi-invariance properties needed for turning the shift mappings into a system of isometries   $(S_i)_{i\in\bz_N}$ in $L^2(\mu)$ satisfying Definition \ref{def2.0} is not well understood. We give in Theorem \ref{th2.5}  an explicit characterization of these measures $\mu$, along with their respective Radon-Nikodym derivatives. We say that these representations are monic because we only need one measure to describe them. (Contrast this with our universal representations in section 4 below.)

The paper is structured as follows:  in section 2, we study {\it monic} representations of the Cuntz algebra $\O_N$ (Definition \ref{def2.4}). We classify them in Theorem \ref{th2.5} and Theorem \ref{th2.9}, and this involves a certain {\it monic system} (Definition \ref{def2.7}) which consists of a measure which are quasi-invariant under shift maps, and some functions which are, up to a phase factor, the roots of the Radon-Nikodym derivatives. The monic representations are called nonnegative if these functions are nonnegative. We prove in Theorem \ref{tha2.10} that two such nonnegative representations are disjoint if and only the associated measures are mutually singular.

In section 3 we present two classes of examples of monic representations: one comes from Markov measures and the other from atomic representations. In Theorem \ref{th3.13}, we prove that different parameters yield disjoint representations and consequently, the Markov measures are mutually singular. We show in Example \ref{ex3.15} that the representations of $\O_N$ that have a one dimensional cyclic $S_i^*$-invariant state are exactly those that are obtained from a monic system with the Kakutani measures \cite{Kak48}.

In section 4, we show how Nelson's universal representation (of an abelian algebra) \cite{Nel69} carries a representation of the Cuntz algebra which is also universal in the sense that it contains all nonnegative monic representation.  

\section{Symbolic dynamics and monic representations}

  Note that our compact infinite product $\K_N$ used below, is a compactification of the $N$-ary tree. The latter, in turn, is a special graph, falling within the graphs called Bratteli diagrams. The Bratteli diagrams in turn serve as useful models, and have a host of applications in symbolic dynamics; see e.g., \cite{Mat11, Ho00}. In fact there is a substantial literature on dynamics in Bratteli diagrams; see e.g., \cite{FrOr13, Kar12, BeKa11, HaYu11}.  (The original paper on Bratteli diagrams is \cite{Bra72}).

But we note that, of the cases in the literature, the question of which systems support a representation of one of the Cuntz algebras has received relatively little attention; see however \cite{BJO04, BJKH01}. It is of interest to find these representations, when they are supported by a symbolic dynamics model. One reason is that when we have an  $\O_N$-representation, the tables can be turned, and we will be able to draw conclusions about the dynamical system from our harmonic analysis of these $\O_N$-representations.

       The cross-road of representations of $C^*$-algebras on the one hand, and dynamics on the other is also of interest for a class of $C^*$-algebras containing the Cuntz algebras, the Cuntz-Krieger algebras  $\O_A$ , and related graph-algebras (\cite{CuKr80, MaPa11, KMST10, BoPe11}; but also here, there has been relatively little activity on determining specific classes of representations of these $\O_A$ and graph-algebras. This is perhaps understandable since, as noted above, already the harmonic analysis for representations of $\O_N$ alone is unwieldy.   

\begin{definition}\label{def2.0}
Let $N\geq 2$. The Cuntz algebra $\O_N$ is the $C^*$-algebra generated by a system of $N$ isometries $(S_i)_{i\in\bz_N}$ satisfying the {\it Cuntz relations}
\begin{equation}
S_i^*S_j=\delta_{ij}I,\quad (i,j\in\bz_N),\quad \sum_{i\in\bz_N}S_iS_i^*=I.
\label{eq2.0.1}
\end{equation}
\end{definition}
\begin{definition}\label{def0.1}
Fix an integer $N\geq 2$. Let $(S_i)_{i\in\bz_N}$ be a representation of the Cuntz algebra $\O_N$ on a Hilbert space $\H$. Let $\bz_N:=\{0,1,\dots,N-1\}$. We will call elements in $\bz_N^k$ {\it words of length $k$}. 
We denote by $\K=\K_N=\bz_N^{\bn}$, the set of all infinite words.
Given two finite words $\alpha=\alpha_1\dots\alpha_n$, $ \beta=\beta_1\dots\beta_m$, we denote by $\alpha\beta$ the concatenation of the two words, so $\alpha\beta=\alpha_1\dots\alpha_n\beta_1\dots\beta_m$. Similarly, for the case when $\beta$ is infinite. Given a word $\omega=\omega_1\omega_2\dots$, and $k$ a non-negative integer smaller than its length, we denote by 
$$\omega|k:=\omega_1\dots\omega_k,$$
the truncated word.

For a finite word $I=i_1\dots i_n$, we denote by 
$$S_I:=S_{i_1}\dots S_{i_n}.$$

We define $\mathfrak A_N$ to be the abelian subalgebra of $\O_N$ generated by $S_IS_I^*$, for all finite words $I$. 
As a $C^*$-algebra, $\mathfrak A_N$ is naturally isomorphic to $C(\K_N)$, the continuous functions on the Cantor group $\K_N$, see Definition \ref{defcantor}.

We say that a subspace $M$ is {\it $S_i^*$-invariant} if $S_i^*M\subset M$ for all $i\in\bz_N$. Equivalently $$P_MS_i^*P_M=S_i^*P_M,$$ where $P_M$ is the projection onto $M$. We say that $M$ is {\it cyclic } for the representation if 
$$\Span\{S_IS_J^*v : v\in M, I,J\mbox{ finite words }\}=\H.$$

\end{definition}

\begin{definition}\label{defcantor}
Fix an integer $N\geq 2$.  The {\it Cantor group on $N$ letters} is 
$$\K=\K_N=\prod_{1}^\infty\bz_N=\{(\omega_1\omega_2\dots) : \omega_i\in \bz_N\mbox{ for all }i=1,\dots\},$$
an infinite Cartesian product.

The elements of $\K_N$ are infinite words.
On the Cantor group, we consider the product topology. We denote by $\B(\K_N)$ the sigma-algebra of Borel subsets of $\K_N$. We denote by $\M(\K_N)$ the set of all finite Borel measures on $\K_N$.

Denote by $\sigma$ the shift on $\K_N$, $\sigma(\omega_1\omega_2\dots)=\omega_2\omega_3\dots$. Define the inverse branches of $\sigma$: for $i\in\bz_N$, $\sigma_i(\omega_1\omega_2\dots)=i\omega_1\omega_2\dots$.

For a finite word $I=i_1\dots i_k\in \bz_N^k$, we define the corresponding {\it cylinder set }

\begin{equation}
\C(I)=\{\omega\in \K_N : \omega_1=i_1,\dots,\omega_k=i_k\}=\sigma_{i_1}\dots\sigma_{i_n}(\K_N).
\label{eqcantor0}
\end{equation}
\end{definition}

\begin{definition}\label{def2.3}
Let $(S_i)_{i\in\bz_N}$ be a representation of the Cuntz algebra $\O_N$ on a Hilbert space $\H$. Then we define the projection value $P$ on the Borel sigma-algebra $\B(\K_N)$ by defining it on cylinders first
\begin{equation}
P(\C(I))=S_IS_I^*\mbox{ for any finite word }I.
\label{eq2.3.1}
\end{equation}
and then extending it by the usual Kolmogorov procedure (see \cite{DHJ13} for details). We call this, the {\it projection valued measure associated to the representation}. 
This projection valued measure then induces a representation $\pi$ of bounded (and of continuous) functions on $\K_N$, by setting
\begin{equation}
\pi(f)=\int_{\K_N} f(\omega)\,dP(\omega).
\label{eq2.3.2}
\end{equation}

For every $x\in \H$, define the Borel measure $\m_x$ on $\K_N$ by
\begin{equation}
\m_x(A)=\ip{x}{P(A)x},\quad(A\in\B(\K_N)).
\label{eq2.3.3}
\end{equation}
Using \eqref{eq2.3.2} and \eqref{eq2.3.3}, and the property
$$P(A\cap B)=P(A)P(B),\quad(A,B\in\B(\K_N)),$$
one obtains
\begin{equation}
\left\|\pi(f)x\right\|_\H^2=\int|f|^2\,d\m_x.
\label{eq2.3.4}
\end{equation}

We will also use the notations
$$P(I)=P(\C(I))\mbox{ for }I=i_1\dots i_n\mbox{ and $P(\omega)=P(\{\omega\})$ for $\omega\in\K_N$}.$$
\end{definition}

\begin{definition}\label{def2.4}
We say that a representation of the Cuntz algebra $\O_N$ on a Hilbert space $\H$ is {\it monic} if there is a cyclic vector $\varphi$ in $\H$ for the abelian subalgebra $\mathfrak A_N$, i.e., 
$$\Span\{S_IS_I^*\varphi : I\mbox{ finite  word }\}=\H.$$
\end{definition}

\begin{definition}\label{def2.7}
A {\it monic system} is a pair $(\mu,(f_i)_{i\in\bz_N})$ where $\mu$ is a finite Borel measure on $\K_N$ and $f_i$ are some functions on $\K_N$ such that $\mu\circ\sigma_i^{-1}\ll\mu$ for all $i\in\bz_N$ and 
\begin{equation}
\frac{d(\mu\circ\sigma_i^{-1})}{d\mu}=|f_i|^2,
\label{eq2.5.1}
\end{equation}
for some functions $f_i\in L^2(\mu)$ with the property that 
\begin{equation}
f_i(x)\neq 0\mbox{ for $\mu$-a.e. $x$ in $\sigma_i(\K_N)$}.
\label{eq2.5.2}
\end{equation}

We say that a monic system is {\it nonnegative} if $f_i\geq0$ for all $i\in\bz_N$.

The representation of $\O_N$ associated to a monic system is 
\begin{equation}
S_if=f_i(f\circ\sigma),\quad(i\in\bz_N,f\in L^2(\mu)).
\label{eq2.5.3} 
\end{equation}

We say that this representation $(S_i)_{i\in\bz_N}$ of the Cuntz algebra is nonnegative if the monic system is. 
\end{definition}

\begin{theorem}\label{th2.5}
Let $(S_i)_{i\in\bz_N}$ be a representation of $\O_N$. The representation is monic if and only if it is unitarily equivalent to a representation associated to a monic system.

\end{theorem}

\begin{proof}
We check that the operators in \eqref{eq2.5.3} define a representation of $\O_N$. 
$$\|S_if\|^2=\int |f_i|^2|f\circ\sigma|^2\,d\mu=\int|f\circ\sigma|^2\,d(\mu\circ\sigma_i^{-1})=\int|f\circ\sigma\circ\sigma_i|^2\,d\mu=\int|f|^2\,d\mu.$$
From \eqref{eq2.5.1}, we have that $f_i$ is supported on $\sigma_i(\K_N)$, and from \eqref{eq2.5.2}, we get that the support of $f_i$ is exactly $\sigma_i(\K_N)$. 

Then 

$$\ip{S_if}{S_jg}=\int \cj f_if_j\cj{(f\circ\sigma)}(g\circ\sigma)\,d\mu=0\mbox{ for $i\neq j$}.$$

We compute $S_i^*$. Define $g_i=\frac{f_i}{|f_i|^2}$ if $f_i(x)\neq 0$, $g_i(x)=0$ if $f_i(x)=0$. Then
$$\ip{S_i^*f}{g}=\int{\cj f}{S_ig}\,d\mu=\int\cj ff_i(g\circ\sigma)\,d\mu=\int \cj f(g\circ\sigma)g_i|f_i|^2\,d\mu=\int (\cj f\circ\sigma_i)(g_i\circ\sigma_i)g\,d\mu$$
so 
\begin{equation}
S_i^*f=(\cj g_i\circ\sigma_i)(f\circ\sigma_i).
\label{eq2.5.4}
\end{equation}
Then
$$\sum_{i\in\bz_N}S_iS_i^*f=\sum_{i\in\bz_N}f_i(\cj g_i\circ\sigma_i\circ\sigma)(f\circ\sigma_i\circ\sigma).$$
For $x\in\sigma_i(\K_N)$, $\sigma_i(\sigma(x))=x$ and $\cj g_i(x)=\frac{\cj f_i(x)}{|f_i(x)|^2}$ (by \eqref{eq2.5.2}). Also $f_j(x)=0$ for $j\neq i$. Therefore 
$$\sum_{i\in\bz_N}S_iS_i^*f(x)=f(x).$$

Thus, we have a representation of $\O_N$. We check that the representation is monic. We have 
$$S_iS_i^*f=f_i(g_i\circ\sigma_i\circ\sigma)(f\circ\sigma_i\circ\sigma)=\chi_{\sigma_i(\K_N)}f.$$
By induction

\begin{equation}
S_IS_I^*f=\chi_{\sigma_{i_1}\dots\sigma_{i_n}(\K_N)}f=\chi_{\C(I)}f\mbox{ for $I=i_1\dots i_n$}.
\label{eq2.5.0}
\end{equation}
Then the constant function $f\equiv 1$ is cyclic for $\mathfrak A_N$, so the representation is monic.

For the converse, if the representation is monic, then let $\varphi$ be a cyclic vector for $\mathfrak A_N$. Define the measure $\mu$ on $\K_N$ by $\mu=\mathfrak m_\varphi$,
$$\ip{\varphi}{\pi(f)\varphi}=\int f\,d\mu,\quad (f\in C(\K_N)).$$
The map $W$ from $C(\K_N)$ to $\H$, $Wf=\pi(f)\varphi$ is linear and isometric so it extends to an isometry from $L^2(\mu)$ to $\H$, onto, because the representation is monic.

Define the operators $\tilde S_i:=W^*S_iW$, $i\in\bz_N$. We check that $\tilde S_i$ are given by \eqref{eq2.5.1}. Let $\tilde S_i1=f_i$. We will use the relations
$$S_i^*\pi(f)S_i=\pi(f\circ\sigma_i),\quad S_i\pi(f)=\pi(f\circ\sigma)S_i,$$
which can be checked first on characteristic functions of cylinder sets. 

 We have 
$$\int |f_i|^2f\,d\mu=\ip{\tilde S_i1}{f\tilde S_i1}_{L^2(\mu)}=\ip{S_i\varphi}{\pi(f)S_i\varphi}_{\H}=\ip{\varphi}{S_i^*\pi(f)S_i\varphi}_{\H}$$$$=\ip{\varphi}{\pi(f\circ\sigma_i)\varphi}_\H=\int f\circ\sigma_i\,d\mu.$$

This implies \eqref{eq2.5.1}. 

For $f\in C(\K_N)$,
$$\tilde S_if=W^*S_iWf=W^*S_i\pi(f)\varphi=W^*\pi(f\circ\sigma)S_i\varphi=W^*\pi(f\circ\sigma)WW^*S_iW1=(f\circ\sigma)f_i.$$

So, we have the formula in \eqref{eq2.5.3}. Then we get as above the formula for $S_i^*$ as in \eqref{eq2.5.4} and the Cuntz relation $\sum S_iS_i^*=I$ implies that the support of $f_i$ must be the entire $\sigma_i(\K_N)$. 
\end{proof}

\begin{proposition}\label{pr1.8}
Let $(\mu,(f_i)_{i\in\bz_N})$ be a monic system. Then $\mu\circ\sigma^{-1}\ll\mu$, and
\begin{equation}
\sum_{j\in\bz_N}\frac{\chi_{\sigma_j(\K_N)}}{|f_j\circ\sigma_j|^2}=\frac{d(\mu\circ\sigma^{-1})}{d\mu}.
\label{eq1.8.1}
\end{equation}

\end{proposition}
\begin{proof}
Let $\varphi$ be a continuous function on $\K_N$. We have
$$\int \varphi\sum_j\frac{\chi_{\sigma_j(\K_N)}}{|f_j\circ\sigma_j|^2}\,d\mu=\int(\varphi\circ\sigma)\sum_j\frac{1}{|f_j|^2}\,d\mu\circ\sigma_j^{-1}=
\int(\varphi\circ\sigma)\sum_j\chi_{\sigma_j(\K_N)}\,d\mu=\int\varphi\circ\sigma\,d\mu$$$$=\int\varphi\,d\mu\circ\sigma^{-1}=\int\varphi\frac{d(\mu\circ\sigma^{-1})}{d\mu}\,d\mu.$$
\end{proof}

\begin{theorem}\label{th2.9}
The representations of $\O_N$ associated to two monic systems $(\mu,(f_i)_{i\in\bz_N})$ and $(\mu',(f_i')_{i\in\bz_N})$ are equivalent if and only if the measures $\mu$ and $\mu'$ are equivalent, i.e. $\mu\ll\mu'$ and $\mu'\ll\mu$, and there exists a function $h$ on $\K_N$ such that 
\begin{equation}
\frac{d\mu'}{d\mu}=|h|^2
\label{eq2.9.1}
\end{equation}
and

\begin{equation}
f_i'=\frac{h\circ\sigma}{h}f_i,\quad(i\in\bz_N)
\label{eq2.9.2}
\end{equation}
\end{theorem}

\begin{proof}
Suppose the two representations are equivalent and let $W:L^2(\mu')\rightarrow L^2(\mu)$ be the intertwining isometric isomorphism. Then the two representations of the abelian subalgebra $\mathfrak A_N$ are equivalent and have cyclic vectors. This implies that the measures are equivalent, and $W$ is a multiplication operator $W=M_h$, for some $h\in L^\infty(\K_N)$. Since $W$ is an isometry, we get
$$\int |f|^2|h|^2\,d\mu=\int |f|^2\,d\mu',\quad (f\in L^2(\mu')).$$
This implies \eqref{eq2.9.1}. Since $W$ is invertible, or since the measures are equivalent, we also get that $h\neq0$, $\mu$-a.e.. From the intertwining property $S_iW=WS_i'$ we obtain
that, for any $f\in L^2(\mu')$,
$$f_i(h\circ\sigma)(f\circ\sigma)=hf_i'(f\circ\sigma).$$
Take $f\equiv1$ and we obtain \eqref{eq2.9.2}.

For the converse, just take $Wf=hf$ on $L^2(\mu')$, and a simple check shows that $W$ is an intertwining isomorphism. 
\end{proof}

\begin{proposition}\label{pr1.10}
Let $(\mu,(f_i)_{i\in\bz_N})$ and $(\mu',(f_i')_{i\in\bz_N})$ be two nonnegative monic systems. If the measures are equivalent and $\sqrt{\frac{d\mu'}{d\mu}}=h$, then \eqref{eq2.9.2} holds. In particular, the two representations of $\O_N$ are equivalent. 

\end{proposition}
\begin{proof}
First, we prove that 
\begin{equation}
\frac{d(\mu'\circ\sigma_i^{-1})}{d(\mu\circ\sigma_i^{-1})}=h^2\circ\sigma. 
\label{eq1.10.1}
\end{equation}

Indeed, if $f$ is a continuous function on $\K_N$, then
$$\int f\,d\mu'\circ\sigma_i^{-1}=\int f\circ\sigma_i\,d\mu'=\int (f\circ\sigma_i)h^2\,d\mu=\int (f\circ\sigma_i)(h^2\circ\sigma\circ\sigma_i)\,d\mu=\int f(h^2\circ\sigma)\,d\mu\circ\sigma_i^{-1}.$$
Then, by the chain rule for Radon-Nikodym derivatives, 
$$f_i'^2=\frac{d\mu'\circ\sigma_i^{-1}}{d\mu'}=\left(\frac{d\mu'\circ\sigma_i^{-1}}{d\mu\circ\sigma_i^{-1}}\right)\left(\frac{d\mu\circ\sigma_i^{-1}}{d\mu}\right)\frac{d\mu}{d\mu'}=(h^2\circ\sigma)f_i^2\frac1{h^2}.$$
Then \eqref{eq2.9.2} follows.
\end{proof}

\begin{proposition}\label{pra2.9}
Let $(\mu,(f_i)_{i\in\bz_N})$, $(\mu',(f_i')_{i\in\bz_N})$ be two nonnegative monic systems. Let $d\mu'=h^2\,d\mu+d\nu$ be the Lebesgue-Radon-Nikodym decomposition, with $h\geq0$ and $\nu$ singular with respect to $\mu$. Then there is a partition of $\K_N$ into Borel sets $\K_N=A\cup B$, such that 
\begin{enumerate}
	\item The function $h$ is supported on $A$, $\nu$ supported on $B$, $\mu(B)=0$, $\nu(A)=0$;
	\item The sets $A,B$ are invariant under $\sigma$, i.e., $\sigma^{-1}(A)=A$, $\sigma^{-1}(B)=B$. 
	\item $\nu\circ\sigma_j^{-1}\ll\nu$, and $k_j:=\sqrt{\frac{d(\nu\circ\sigma_j^{-1})}{d\nu}}$ is supported on $B$.
	\item $f_j'h=f_j(h\circ\sigma)$ on $A$ and $f_j'=k_j$ on $B$. 
\end{enumerate}

\end{proposition}
\begin{proof}
Let $\tilde B$ be a support of $\nu$, such that $\mu(\tilde B)=0$. Since $\mu$ is part of a monic system, from Proposition \ref{pr1.8}, it follows that $\sigma^{-1}(\tilde B)$ and $\sigma_j^{-1}(\tilde B)$ have $\mu$-measure zero. Therefore we can take the orbit $B$ of $\tilde B$ under $\sigma$ and $\sigma_j$ and this will have $\mu$ measure zero. Let $A:=\K_N\setminus B$. Then this is a support for $\mu$ and we can chose $h$ to be supported on $A$ and $\nu(A)=0$. Also $A$ and $B$ are invariant under $\sigma$. 

To prove (iii), let $E$ be a Borel set with $\nu(E)=0$. Then $\nu(E\cap B)=0$ so $\mu'(E\cap B)=0$. Then $\mu'(\sigma_j^{-1}(E\cap B))=0$ which means that $\mu'(\sigma_j^{-1}(E)\cap B)=0$, so 
$\nu(\sigma_j^{-1}(E))=0$. Since $B$ is invariant and $\nu$ and $\nu\circ\sigma_j^{-1}$ are supported on $B$, it follows that $k_j$ is supported on $B$. 

For (iv), let $f$ be a bounded Borel function supported on $A$. Then 
$$\int f_j'^2fh^2\,d\mu=\int f_j'^2f\,d\mu'=\int f\circ\sigma_j\,d\mu'=\int (f\circ\sigma_j)h^2\,d\mu=\int (f\circ\sigma_j)(h^2\circ\sigma\circ\sigma_j)\,d\mu$$$$=\int f(h^2\circ\sigma)\,d\mu\circ\sigma_j^{-1}=\int f(h^2\circ\sigma)f_j\,d\mu$$
This implies the first relation.

For the second, take $f$ supported on $B$ and the result follows from the fact that $\mu'|B=\nu$. 

\end{proof}

\begin{theorem}\label{tha2.10}
Let $(\mu,(f_i)_{i\in\bz_N})$ and $(\mu',(f_i')_{i\in\bz_N})$ be two nonnegative monic systems. Then the two associated representations of $\O_N$ are disjoint if and only if the measures $\mu$ and $\mu'$ are mutually singular. 
\end{theorem}

\begin{proof}
If the measures are mutually singular, let $W:L^2(\mu)\rightarrow L^2(\mu')$ be an intertwining operator. Then $W$ also intertwines the two representations of the abelian subalgebra $\mathfrak A_N$. But since these representations are given by multiplication operators (see \eqref{eq2.5.0}), and the measures are mutually singular, it follows that $W=0$.

For the converse, assume that the representations are disjoint and that the measures are not mutually singular. Then, use Proposition \ref{pra2.9} and decompose $d\mu'=h^2\,d\mu+d\nu$, with the subsets $A,B$ as in Proposition \ref{pra2.9}.

Define the operator $W$ on $L^2(\mu')$ by $Wf=fh$ if $f\in L^2(A)$, and $Wf=0$ on the orthogonal complement of $L^2(A)$. Since $A$ is invariant under $\sigma$, $L^2(A)$ is a reducing subspace for the representation. We check that $W$ is intertwining, using Proposition \ref{pra2.9}(iv):
$$S_jWf=f_j(h\circ\sigma)(f\circ\sigma)=f_j'h(f\circ\sigma)=S_j'Wf.$$

\end{proof}

\begin{theorem}\label{th2.10}
Let $(\mu,(f_i)_{i\in\bz_N})$ be a monic system and let $(S_i)_{i\in\bz_N}$ be the associated representation of $\O_N$. Then the commutant of the representation consists of multiplication operators by functions $h$ with $h\circ\sigma=h$, $\mu$-a.e.. In particular, the representation is irreducible if and only if $\sigma$ is ergodic with respect to the measure $\mu$, i.e., the only Borel set $A$ in $\K_N$ with $\sigma^{-1}(A)=A$ are sets of measure zero, or of full measure.

\end{theorem}
  
  \begin{proof}
  Let $T$ be an operator in the commutant. Then $T$ commutes with the representation of the abelian subalgebra $\mathfrak A_N$. Since this has a cyclic vector, it is a maximal abelian subalgebra consisting of multiplication operators (see \eqref{eq2.5.0}). Then $T$ must be a multiplication operator $T=M_h$ with $h\in L^\infty(\mu)$. Since $T$ commutes with $S_i$ we obtain
  $$hf_i(f\circ \sigma)=f_i(h\circ\sigma)(f\circ\sigma),\quad(f\in L^2(\mu)).$$
  Take $f\equiv1$ and use \eqref{eq2.5.2} to conclude that $h\circ\sigma=h$.
  
  Conversely, it is easy to see that any such multiplication operator commutes with the isometries $S_i$. 
  
  \end{proof}
  
  \section{Classes of monic representations}

\subsection{Markov measures}

       We describe here a class of measures which give rise to monic systems and therefore to monic representations of $\O_N$.        
     We suggest Markov processes as a tool, and the corresponding measures will be families of Markov measures.  Given a system of Markov transition probability matrices, the corresponding Markov measure  $\mu$  will then be constructed with the use of Kolmogorov's consistency condition; see e.g., \cite{Pet89, Jor06}.)  Specifically, a Markov process in symbolic dynamics is specified by a system of Markov transition matrices, and from this one then construct associated Markov measures  $\mu$ on $\K_N$ (= the set of infinite paths.) Since monic representations of $\O_N$  are hard to come by, we begin our approach here via Markov processes picking out a rather special system of Markov transition matrices. This will facilitate explicit formulas and avoid some delicate questions regarding infinite products. 
      
      We stress our two sources of motivation; one is the above mentioned list of applications, and the other is two classical papers on infinite products.   
Since Markov measures use both finite and infinite products, a rigorous analysis depends on infinite products, for example Kolmogorov's inductive limit construction. In this connection we have been motivated by two classics,  \cite{vN39} by von Neumann, and \cite{Kak48} by Kakutani. The latter \cite{Kak48}  has the dichotomy theorem for infinite product measures (by Kakutani). But our related use of sigma-measures  in section 4 below is also motivated by \cite{Kak48}. Here we use sigma-measures in our construction of a special representation of $\O_N$ having certain universality properties, Theorem \ref{th2.2}.
 
\begin{definition}\label{def3.8}
A class of Markov measures on $\K_N$ are defined by a vector $\lambda=(\lambda_0,\dots,\lambda_{N-1})$ and an $N\times N$ matrix $T$ such that $\lambda_i>0$, $T_{ij}>0$ for all $i,j\in\bz_N$ and if $e=(1,1\dots,1)^T$ then
\begin{equation}
\lambda T=\lambda\mbox{ and }Te=e.
\label{eq3.8.1}
\end{equation}

  Then there exists a unique Borel measure on $\K_N$ such that
  \begin{equation}
\mu(\C(I))=\lambda_{i_1}T_{i_1,i_2}\dots T_{i_{n-1},i_n}\mbox{ if }I=i_1\dots i_n.
\label{eq3.8.2}
\end{equation}
  \end{definition}

\begin{proposition}\label{pr3.9}
For the Markov measure $\mu$ defined above, $\mu\circ\sigma_j^{-1}\ll\mu$ and 
\begin{equation}
\frac{d(\mu\circ\sigma_j^{-1})}{d\mu}(x_1x_2\dots)=\delta_{j,x_1}\frac{\lambda_{x_2}}{\lambda_j T_{j,x_2}}.
\label{eq3.9.1}
\end{equation}

\end{proposition}

\begin{proof}
Let $I=i_1\dots i_n$. We have
$$\mu\circ\sigma_j^{-1}(\C(I))=\delta_{j,i_1}\mu(\{(x_1x_2\dots) : x_1=i_2,x_2=i_3,\dots, x_{n-1}=i_n\})=\delta_{j,i_1}\lambda_{i_2}T_{i_2,i_3}T_{i_3,i_4}\dots T_{i_{n-1},i_n}.$$

Now let 
$$F_j(x_1x_2\dots)=\delta_{j,x_1}\frac{\lambda_{x_2}}{\lambda_j T_{j,x_2}}.$$
Then, note that $F_j$ is constant on $\C(I)$ (we assume $n\geq 2$), so
$$\int_{\C(I)}F_j(x)\,d\mu(x)=\delta_{j,i_1}\frac{\lambda_{i_2}}{\lambda_j T_{j,i_2}}\lambda_{i_1}T_{i_1,i_2}T_{i_2,i_3}\dots T_{i_{n-1},i_n}$$
$$=\delta_{j,i_1}\lambda_{i_2}T_{i_2,i_3}T_{i_3,i_4}\dots T_{i_{n-1},i_n}=\mu\circ\sigma_j^{-1}(\C(I)).$$
This proves \eqref{eq3.9.1}.
\end{proof}

With Theorem \ref{th2.5} and Proposition \ref{pr3.6} below we obtain
\begin{corollary}\label{cor3.10}
Let $\mu$ be a Markov measure as above. 
Let $f_j$ be the functions on $\K_N$ defined by
\begin{equation}
f_j(x_1x_2\dots)=\delta_{j,x_1}\sqrt{\frac{\lambda_{x_2}}{\lambda_j T_{j,x_2}}}.
\label{eq3.10.1}
\end{equation}
Let $S_j$ be the operators on $L^2(\mu)$ defined by
\begin{equation}
S_jf=f_j(f\circ\sigma).
\label{eq3.10.2}
\end{equation}
Then $(S_j)_{j\in\bz_N}$ defines a monic representation of $\O_N$ which can be embedded isometrically into the universal representation (from section 4).
\end{corollary}

\begin{lemma}\label{lem3.11}
The Markov measure $\mu$ is invariant for $\sigma$, i.e. 
\begin{equation}
\mu\circ\sigma^{-1}=\mu.
\label{eq3.11.1}
\end{equation}
\end{lemma}

\begin{proof}
It is enough to check \eqref{eq3.11.1} on cylinder sets. Let $I=i_1\dots i_n$. We have, using $\lambda T=\lambda$,
$$\mu(\sigma^{-1}(\C(I))=\sum_{i\in\bz_N}\mu(\C(iI))=\sum_{i\in \bz_N}\lambda_iT_{i,i_1}T_{i_1,i_2}\dots T_{i_{n-1},i_n}=\lambda_{i_1}T_{i_1,i_2}\dots T_{i_{n-1},i_n}=\mu(\C(I)).$$
\end{proof}

\begin{lemma}\label{lem3.12}
Let $M$ be the subspace of functions in $L^2(\mu)$ that depend only on the first coordinate. Then $M$ is $S_i^*$-invariant and cyclic for the representation of $\O_N$. 

\end{lemma}

\begin{proof}
Let $f$ be a function in $M$. We have:
$$g_j(x_1x_2\dots)=\frac{f_j}{|f_j|^2}(x_1x_2\dots)=\delta_{j,x_1}\sqrt{\frac{\lambda_jT_{j,x_2}}{\lambda_{x_2}}}.$$
Then, using \eqref{eq2.5.4},
$$S_j^*f(x_1x_2\dots)=\cj g_j(jx_1x_2\dots)f(jx_1x_2\dots)=\sqrt{\frac{\lambda_jT_{j,x_1}}{\lambda_{x_1}}}f(jx_1x_2\dots),$$
so $S_j^*f$ also depends only on the first coordinate. Therefore $M$ is $S_i^*$-invariant. 

To see that $M$ is cyclic, note that the function $1$ is cyclic for the set of operators $S_IS_I^*$. But $S_I^*1$ is in $M$ and therefore the vectors $S_IS_I^*1$ are in $S_IM$ and they span $L^2(\mu)$. 
\end{proof}

\begin{theorem}\label{th3.12}
The representation of $\O_N$ associated to a Markov measure $\mu$, as in Corollary \ref{cor3.10}, is irreducible. 

\end{theorem}

\begin{proof}
We use a result from \cite{BJKW00}, see also \cite[Theorem 5.5]{DHJ13}, which states the following: since we have the $S_i^*$-invariant cyclic subspace $M$, to verify that the representation is irreducible, let $V_i^*=S_i^*P_M$, where $P_M$ is the projection onto $M$, and the only solutions to the equation 
\begin{equation}
\sum_{i\in\bz_N}V_iXV_i^*=X
\label{eq3.12.1}
\end{equation}
should be constant multiples of the identity operator. 

The subspace $M$ has the following orthonormal basis 
$$e_j(x_1x_2\dots)=\frac{1}{\sqrt{\lambda_j}}\delta_{j,x_1},\quad (j\in\bz_N).$$
We compute 
$$V_i^*e_j(x_1x_2\dots)=S_i^*e_j(x_1x_2\dots)=\sqrt{\frac{\lambda_i T_{i,x_1}}{\lambda_j\lambda_{x_1}}}\delta_{j,i},$$
so
$$V_i^*e_j=\delta_{i,j}\sum_{x_1\in\bz_N}\sqrt{ T_{i,x_1}}e_{x_1}.$$
Therefore the matrix of $V_i^*$ in this orthonormal basis is
\begin{equation}
\left(\delta_{j,i}\sqrt{ T_{i,k}}\right)_{k,j\in\bz_N},
\label{eq3.12.2}
\end{equation}
so it has only one non-zero column on position $i$. 

Let $v_i=(\sqrt{{ T_{i,j}}})_{j\in\bz_N}^T$.

Let $X$ be a solution for \eqref{eq3.12.1}. Then the matrix of $V_iXV_i^*$ has only one non-zero entry on the $i$-th position of the diagonal, and that is equal to $\ip{v_i}{Xv_i}$. Thus, the matrix $X$ has to be diagonal and we have  
$$\ip{v_i}{Xv_i}=X_{i,i}.$$
This implies that, for all $i\in\bz_N$,
$$\sum_{j\in\bz_N}{ T_{i,j}X_{j,j}}=X_{i,i}.$$

But this means that the vector $(X_{i,i})_{i\in \bz_N}$ is an right-eigenvector for $T$ with eigenvalue 1, and since the entries of $T$ are positive, the Perron-Frobenius theorem (see also Remark \ref{rem3.14}) implies that $X_{i,i}=c$ for all $i$ for some constant $c$. So $X=cI$ and the representation is irreducible.
\end{proof}

\begin{remark}
Another way of proving the fact that the representation of $\O_N$ associated with a Markov measure is irreducible, by Theorem \ref{th2.10}, is by showing that $\sigma$ is ergodic with respect to $\mu$. This fact is well known, see e.g. \cite{Pet89}. But the converse also holds, so our proof shows also that $\sigma$ is ergodic with respect to $\mu$. 

\end{remark}

\begin{corollary}\label{cor3.13}
The shift $\sigma$ is ergodic with respect to the Markov measure $\mu$.
\end{corollary}

\begin{proof}
By Theorem \ref{th3.12} the representation of $\O_N$ associated with the Markov measure $\mu$ is irreducible. The result folows then from Theorem \ref{th2.10}.

\end{proof}

\begin{theorem}\label{th3.13}
Let $\mu$, $\mu'$ be two Markov measures associated to $(T,\lambda)$ and $(T',\lambda')$ respectively. If $T\neq T'$ then the two representations of $\O_N$ are disjoint. Consequently, the two measures are mutually singular. 

\end{theorem}

\begin{proof}
As in the proof of Theorem \ref{th3.12}, we use the result in \cite{BJKW00}. We have the $S_i^*$-invariant cyclic subspaces $M$, $M'$ of $L^2(\mu)$ and $L^2(\mu')$ consisting of functions which depend only on the first coordinate. Define $V_i^*=S_i^*P_M$ and $V_i'^*=S_i'^*P_{M'}$. The representations are disjoint if and only if the only solution for 
\begin{equation}
\sum_{i\in\bz_N}V_i'XV_i^*=X,
\label{eq3.13.1}
\end{equation}
are multiples of the identity.

We use the same orthonormal basis for $M$ and similarly for $M'$ and we have that the matrix of $V_i^*$ is given in \eqref{eq3.12.2}, similarly for $V_i'^*$. Then the equation \eqref{eq3.13.1} implies that $X$ has to be a diagonal matrix and 
$$\sum_{j\in\bz_N}\sqrt{T'_{i,j}}X_{j,j}\sqrt{T_{i,j}}=X_{i,i},\quad(i\in\bz_N).$$

Let $i$ be such that $|X_{i,i}|=\max_k|X_{k,k}|$. We have, using the Schwarz inequality:
$$|X_{i,i}|\leq \sum_{j\in\bz_N}\sqrt{T'_{i,j}}|X_{j,j}|\sqrt{T_{i,j}}\leq\left(\sum_{j\in\bz_N}T'_{i,j}\right)^{\frac12}\left(\sum_{j\in\bz_N}|X_{j,j}|^2T_{i,j}\right)^{\frac12}$$$$\leq |X_{i,i}|\left(\sum_{j\in\bz_N}T_{i,j}\right)^{\frac12}=|X_{i,i}|.$$
Therefore, we must have equalities in all inequalities. So $|X_{j,j}|=|X_{i,i}|$ for all $j$. Also, we have equality in the Schwarz inequality and this means that the vectors $(T_{i,k})_k$ and $(|X_{k,k}|T_{i,k}')_k$ are proportional. Since the sum of their components is the same, the two vectors are equal. So because $T\neq T'$, we get $|X_{k,k}|=0$. Therefore $X=0$.  

The last statement follows from Theorem \ref{tha2.10}.

\end{proof}

\begin{remark}\label{rem3.14}
In the proof of Theorem \ref{th3.12} we used the Perron-Frobenius theorem to conclude that $X$ is a multiple of the identity. But this is not really needed; the argument used in the proof of Theorem \ref{th3.13} can be used instead: since we have equality in the first triangle inequality, it follows that $X_{k,k}=c|X_{k,k}|$ for all $k$ for some constant $c$. So $X_{k,k}$ is constant and $X$ is a multiple of the identity.
\end{remark}

\begin{example}\label{ex3.15}
Let $z_i$, $i\in\bz_N$ be some complex numbers with $\sum_i|z_i|^2=1$. Let $p_i=|z_i|^2$, $i\in\bz_N$. Define the matrix 
$$T=\begin{pmatrix}
	p_1&p_2&\dots&p_n\\
	p_1&p_2&\dots&p_n\\
	\vdots&\vdots&\ddots&\vdots\\
	p_1&p_2&\dots&p_n
\end{pmatrix}$$
Let $\lambda_i=p_i$ for all $i\in\bz_N$. 
The associated Markov measure for these particular parameters $T$ and $\lambda$ is the Kakutani measure \cite{Kak48}
$$\mu(\C(i_1\dots i_n))=p_{i_1}\dots p_{i_n},\quad(i_1,\dots,i_n\in\bz_N).$$
It satisfies the invariance equation
\begin{equation}
\int f\,d\mu=\sum_{i\in\bz_N}p_i\int f\circ\sigma_i\,d\mu,\quad (f\in C(\K_N)).
\label{eq3.15.1}
\end{equation}
(This can be checked first on characteristic functions of cylinder sets).

Define the functions 
\begin{equation}
f_i=\frac{1}{z_i}\chi_{\sigma_j(\K_N)},\quad(i\in\bz_N)
\label{eq3.15.2}
\end{equation}

(Note that the absolute value $|f_i|$ matches the formula in \eqref{eq3.10.1}).

We check that the equation \eqref{eq2.5.1} is satisfied. We have from \eqref{eq3.15.1}, for $f\in C(\K_N)$,\
$$\int f\chi_{\sigma_i(\K_N)}\,d\mu=\sum_jp_j(\chi_{\sigma_i(\K_N)}\circ\sigma_j)(f\circ\sigma_j)\,d\mu=p_i\int(f\circ\sigma_i)\,d\mu=p_i\int f\,d(\mu\circ\sigma_i^{-1}).$$
This implies \eqref{eq2.5.1}.

Therefore, by Theorem \ref{th2.5}, the operators
$$S_if=f_i(f\circ\sigma),\quad(i\in\bz_N,f\in L^2(\mu)),$$
define a monic representation of $\O_N$.

The function $\varphi=1$ is cyclic for $\mathfrak A_N$. Also, with \eqref{eq2.5.4}, we have $g_i=z_i\chi_{\sigma_i(\K_N)}$ and
$$S_i^*\varphi=z_i(\chi_{\sigma_i(\K_N)}\circ\sigma_i)(\varphi\circ\sigma_i)=z_i\varphi.$$

Thus the one-dimensional space $M$ spanned by $\varphi$ is $S_i^*$-invariant and cyclic for the representation of $\O_N$.

\end{example}

\begin{corollary}\label{cor3.16}
Let $(S_i)_{i\in\bz_N}$ be a representation of $\O_N$ on a Hilbert space $\H$. Suppose that there is a cyclic vector $\varphi\in\H$ and some complex numbers $z_i$, $i\in\bz_N$ such that $S_i^*\varphi=z_i\varphi$ for all $i\in\bz_N$. 
Then the representation is monic.
\end{corollary}

\begin{proof}
The Cuntz relations imply that $\sum_i|z_i|^2=1$. 
The existence of such a cyclic $S_i^*$-invariant space determines completely the representation (see \cite[Theorem 5.1]{BJKW00}), therefore this representation is equivalent to the one in Example \ref{ex3.15}, so it is monic.

\end{proof}

\begin{remark}\label{rem3.13}
Let $(S_i)_{i\in\bz_N}$ be a representation of $\O_N$ on some Hilbert space $\H$. Suppose there exists a closed $S_i^*$-invariant, cyclic space $M$. Let $P_M$ be the projection onto $M$. Define the operators 
\begin{equation}
V_i^*=S_i^*P_M=P_MS_i^*P_M.
\label{eqa3.13.1}
\end{equation}
Then, by \eqref{eq2.0.1} these operators satisfy the equation
\begin{equation}
\sum_{i\in\bz_N}V_iV_i^*=I_M.
\label{eqa3.13.2}
\end{equation}

It is known that a converse also holds (see e.g., \cite{BJKW00}): if some operators $V_i$ are given on a space $M$, satisfying \eqref{eqa3.13.2}, then there exists a bigger Hilbert space $\H$ and a representation $(S_i)_{i\in\bz_N}$ of $\O_N$ such that $M$ is $S_i^*$-invariant and cyclic and \eqref{eqa3.13.1} holds. Moreover this representation is unique up to unitary equivalence. 

It is natural to ask what kind of operators $(V_i)_{i\in \bz_N}$ on some Hilbert space $M$ would yield {\it monic} representations with the result mentioned above. We have here two examples: Corollary \ref{cor3.16} shows that one-dimensional spaces $M$ always yield monic representations, associated with Kakutani measures, and secondly, the transpose of the matrices in \eqref{eq3.12.2} also yield monic representations associated with Markov measures, as we can see from the proof of Theorem \ref{th3.12}.

\end{remark}

%
%
%
%
%
%

\subsection{Atomic representations}

Recall some notions from \cite{DHJ13}:
\begin{definition}\label{def4.1}
A representation $(S_i)_{i\in\bz_N}$ of $\O_N$ is called {\it atomic} if there exist a subset $\Omega$ of $\K_N$ such that 
$$\sum_{\omega\in\Omega}P(\omega)=I.$$
If $\omega$ is an element of $\K_N$ such that $P(\omega)\neq 0$, then $\omega$ is called an {\it atom}.
\end{definition}

\begin{theorem}\label{th4.2}
Let $(S_i)_{i\in \bz_N}$ be an atomic representation of $\O_N$ on a separable Hilbert space $\H$. Then the representation is monic if and only if, for every atom $\omega\in\K_N$, $P(\omega)\H$ is one-dimensional. In this case, the associated measure from the monic system is atomic and the set of atoms is countable.

\end{theorem}

\begin{proof}
Suppose the representation is monic and there is an atom $\omega$ such that $P(\omega)\H$ has dimension bigger than 2. Let $v_1,v_2$ be two unitary orthogonal vectors in $P(\omega)\H$, and let $P_{v_1}$ and $P_{v_2}$ be the corresponding orthogonal projections. Then it is easy to see that $P_{v_1}$ and $P_{v_2}$ commute with $P(\omega')$ for all atoms $\omega'$. Then, they commute with $P(A)$ for any Borel subset $A$ of $\K_N$. 

But, from Theorem \ref{th2.5} and its proof, we see that $P(A)$ can be considered as multiplication operators (see \eqref{eq2.5.0}), and they have a cyclic vector, so they form a maximal abelian subalgebra, and therefore any operator that commutes with $P(A)$, for all $A$, must be a multiplication operator. But $P_{v_1}$ and $P_{v_2}$ are not. 

The measure $\mu$ from the monic system is supported on the union of all atoms, and therefore it is atomic.

For the converse, we have that there is a sequence $\{\omega_n\}_{n\in\bn}$ of atoms with $\sum_nP(\omega_n)=I$ and some unitary vectors $e_n$ such that $e_n$ spans $P(\omega_n)$. Then $\{e_n :n\in\bn\}$ is an orthonormal basis for $\H$. Define $\varphi=\sum_n\frac{1}{2^n}e_n$. Then $P(\omega_n)\varphi=\frac1{2^n}e_n$ so 
$$e_n\in\Span\{S_IS_I^*=P(I) : I\mbox{ finite word}\}.$$
This implies that the representation is monic.

\end{proof}

\section{A universal representation}
  
  We recall some facts from \cite{Nel69}. The Hilbert space of $\sigma$-functions on $\K_N$ is constructed as follows.
  
  \begin{definition}\label{defuniv}
   We write $(f,\mu)$ for a pair with $f$ in $L^2(\mu)$ and $\mu$ a finite Borel measure on $\K_N$. We say that $(f,\mu)$ and $(g,\nu)$ are equivalent $(f,\mu)\sim(g,\nu)$ if there exists some measure $\lambda$ such that $\mu\ll\lambda$, $\nu\ll\lambda$ and 
  $$f\sqrt{\frac{d\mu}{d\lambda}}=g\sqrt{\frac{d\nu}{d\lambda}},\mbox{ $\lambda$-a.e.}.$$
  The equivalence class of a pair $(f,\mu)$ is denoted $f\sqrt{d\mu}$ and is called a $\sigma$-function. The set of all $\sigma$-functions is denoted $\H(\K_N)$ and it is a Hilbert space with addition defined by
  $$f\sqrt{d\mu}+g\sqrt{d\nu}=\left(f\sqrt{\frac{d\mu}{d\lambda}}+g\sqrt{\frac{d\nu}{d\lambda}}\right)\sqrt{d\lambda}$$
  where $\mu\ll\lambda$, $\nu\ll\lambda$, and the inner product defined by
  $$\ip{f\sqrt{d\mu}}{g\sqrt{d\nu}}=\int \cj f g\sqrt{\frac{d\mu}{d\lambda}}\sqrt{\frac{d\nu}{d\lambda}}\,d\lambda.$$

  We define now the {\it universal isometries} $S_i$ on $\H(\K_N)$ by
  \begin{equation}
S_i(f\sqrt{d\mu})=(f\circ\sigma)\sqrt{d\mu\circ\sigma_i^{-1}}.
\label{eq7.1}
\end{equation}

\end{definition}
  
  \begin{proposition}\label{pr7.2}
  The isometries $(S_i)_{i\in\bz_N}$ in \eqref{eq7.1} define a representation of the Cuntz algebra $\O_N$. The adjoints are given by
  \begin{equation}
S_i^*(f\sqrt{d\mu})=f\circ\sigma_i\sqrt{d(\mu|_{\sigma_i(\K_N)}\circ\sigma^{-1})},
\label{eq7.2.1}
\end{equation}
where $\mu|_A$ indicates the restriction of the measure $\mu$ to the Borel set $A$, i.e., $\mu|_A(B)=\mu(A\cap B)$ for all $B\in\B(\K_N)$. 

The associated projection valued map $P$ is given by
\begin{equation}
P(A)(f\sqrt{d\mu})=\chi_A f\sqrt{d\mu}.
\label{eq7.2.2}
\end{equation}
  \end{proposition}

  \begin{proof}
  First we have to check that $S_i$ is well defined. But if $\lambda$ implements the equivalence of $(f,\mu)$ and $(g,\nu)$ then $\lambda\circ \sigma_i^{-1}$ implements the equivalence between $(f\circ\sigma, \mu\circ\sigma_i^{-1})$ and $(g\circ\sigma,\nu\circ\sigma_i^{-1})$; use also the fact that
  $$\frac{d\mu\circ\sigma_i^{-1}}{d\lambda\circ\sigma_i^{-1}}=\frac{d\mu}{d\lambda}\circ\sigma.$$ The linearity and isometry property are also easily checked.
 
 Next, we derive \eqref{eq7.2.1}. Note that $\sigma\circ\sigma_i=1_{\K_N}$ and $\sigma_i\circ\sigma|_{\sigma_i(\K_N)}=1_{\sigma_i(\K_N)}$. We have
 $$\ip{S_i(f\sqrt{d\mu})}{g\sqrt{d\nu}}=\int \cj f\circ\sigma\sqrt{\frac{d(\mu\circ\sigma_i^{-1})}{d\lambda}}g\sqrt{\frac{d\nu}{d\mu}}\,d\lambda$$
 where $\mu\circ\sigma_i^{-1}\ll\lambda$, $\nu\ll\lambda$.

 Note that $\mu\circ\sigma_i^{-1}$ is supported on $\sigma_i(\K_N)$. So $\frac{d(\mu\circ\sigma_i^{-1})}{d\lambda}$ is supported on $\sigma_i(\K_N)$, therefore we can restrict all functions in the previous relation to $\sigma_i(\K_N)$ and we further have
 $$=\int \cj f\circ\sigma\sqrt{\frac{d(\mu\circ\sigma_i^{-1})}{d\lambda|_{\sigma_i(\K_N)}}}g|_{\sigma_i(\K_N)}\sqrt{\frac{d\nu|_{\sigma_i(\K_N)}}{d\lambda|_{\sigma_i(\K_N)}}}d\lambda
 $$$$=\int(\cj f\circ\sigma)(g\circ\sigma_i\circ\sigma)|_{\sigma_i(\K_N)}\sqrt{\left(\frac{d(\mu\circ\sigma_i^{-1})}{d(\lambda\circ\sigma^{-1}\circ\sigma_i^{-1})}\right)\left(\frac{d(\nu|_{\sigma_i(\K_N)}\circ\sigma^{-1}\circ\sigma_i^{-1})}{d(\lambda\circ\sigma^{-1}\circ\sigma_i^{-1})}\right)}d\lambda
 $$
 $$
 =\int (\cj f\circ\sigma) g\circ\sigma_i\circ\sigma\sqrt{\left(\frac{d\mu}{d(\lambda\circ\sigma^{-1})}\circ\sigma\right)\left(\frac{d(\nu|_{\sigma_i(\K_N)}\circ\sigma^{-1})}{d(\lambda\circ\sigma^{-1})}\circ\sigma\right)}d\lambda
 $$
 $$
 =\int \cj f\, g\circ\sigma_i\sqrt{\left(\frac{d\mu}{d(\lambda\circ\sigma^{-1})}\right)\left(\frac{d(\nu|_{\sigma_i(\K_N)}\circ\sigma^{-1})}{d(\lambda\circ\sigma^{-1})}\right)}d(\lambda\circ\sigma^{-1}).
 $$
 Since $\mu\circ\sigma_i^{-1}\ll\lambda$, we have $\mu\ll\lambda\circ\sigma^{-1}$. Also $\nu|_{\sigma_i(\K_N)}\circ\sigma^{-1}\ll\lambda\circ\sigma^{-1}$. Then
 $$=\ip{f\sqrt{d\mu}}{g\circ\sigma_i\sqrt{d(\nu|_{\sigma_i(\K_N)}\circ\sigma^{-1})}}.$$
 This proves \eqref{eq7.2.1} and with this the Cuntz relations are easy to check.

 For \eqref{eq7.2.2}, we can check that the operator $P'(A)$ defined by the right-hand side of \eqref{eq7.2.2} satisfies
 \begin{equation}
P'(\sigma_i(A))=S_iP'(A)S_i^*,
\label{eq7.2.3}
\end{equation}
for all Borel subsets $A$ of $\K$ and all $i\in\bz_N$.
But $$S_iP'(A)S_i^*(f\sqrt{d\mu})=(\chi_A\circ \sigma) (f\circ\sigma_i\circ\sigma)\sqrt{d(\mu|_{\sigma_i(\K_N)}\circ\sigma^{-1}\circ\sigma_i^{-1})}$$
$$=
\chi_{\sigma_i(\K_N)}(\chi_A\circ \sigma) (f\circ\sigma_i\circ\sigma)\sqrt{d\mu}
=\chi_{\sigma_i(A)}f\sqrt{d\mu}=P'(\sigma_i(A))(f\sqrt{d\mu})$$
Since $P'(\K_N)=I$, \eqref{eq7.2.3}  implies that $P=P'$ on every cylinder, and hence on every Borel set. 
 
  \end{proof}

\begin{theorem}\label{th2.2}
An operator $T$ on $\H(\K_N)$ commutes with the representation  $\pi_{universal}(\mathfrak A_N)$ if and only if for each measure $\mu\in \mathcal M(\K_N)$ there exist a function $F_\mu$ in $L^\infty(\mu)$ with the following properties:

\begin{enumerate}
	\item $\sup_\mu\|F_\mu\|_{L^\infty(\mu)}<\infty$. 
	\item If $\mu\ll\lambda$ then $F_\mu=F_\lambda$, $\mu$-a.e.
	\item $T(f\sqrt{d\mu})=F_\mu f\sqrt{d\mu}$ for all $f\sqrt{d\mu}\in \H(\K_N)$. 
\end{enumerate}

Moreover $T$ commutes with $\pi_{universal}(\O_N)$ if and only if for every $\mu\in\M(\K_N)$
\begin{equation}
F_\mu=F_{\mu\circ\sigma_i^{-1}}\circ\sigma_i,\quad\mu\mbox{-a.e.}\quad(i\in\bz_N).
\label{eq2.2.1}
\end{equation}

\end{theorem}

\begin{proof}
Suppose $T$ is an operator that commutes with the representation of $\mathfrak A_N$. Then $T$ commutes with the projection valued measure $P$. 
\begin{lemma}\label{lem2.3}
If $T$ commutes with the representation of $\mathfrak A_N$, then for every $x\in\H(\K_M)$, $\m_{Tx}\ll \m_x$.
\end{lemma}

\begin{proof}
We have for each Borel set $A$:
$$\m_{Tx}(A)=\ip{Tx}{P(A)Tx}=\ip{T^*Tx}{P(A)x}\leq \|T^*Tx\|\|P(A)x\|^2=\|T^*Tx\|\m_x(A).$$
This implies that $\m_{Tx}\ll\m_x$.
\end{proof}

\begin{lemma}\label{lem2.3.1}
For every $f\sqrt{d\mu}$ in $\H(\K_N)$
\begin{equation}
d\m_{f\sqrt{d\mu}}=|f|^2d\mu
\label{eq2.3.1.1}
\end{equation}
\end{lemma}

\begin{proof}
Using Proposition \ref{pr7.2}, we have
$$\m_{f\sqrt{d\mu}}(A)=\ip{f\sqrt{d\mu}}{P(A)(f\sqrt{d\mu})}=\ip{f\sqrt{d\mu}}{\chi_Af\sqrt{d\mu}}=\int\chi_A|f|^2\,d\mu. $$
This proves the lemma. 
\end{proof}

\begin{lemma}\label{lem2.4}\cite{Nel69}.
For every finite Borel measure $\mu$ on $\K_N$, define the operator $W_\mu$ from $L^2(\mu)$ to $\H(\K_N)$, 
\begin{equation}
W_\mu(f)=f\sqrt{d\mu}.
\label{eq2.4.1}
\end{equation}
Then $W_\mu$ is an isometry into a subspace of $\H(\K_N)$ which we denote by $\L^2(\mu)$.
For any bounded measurable function $f$ on $\K_N$ define the multiplication operator $M_fg=fg$, $g\in L^2(\mu)$. Then 
\begin{equation}
W_\mu M_f=\pi(f)W_\mu
\label{eq2.4.2}
\end{equation}
\end{lemma}

\begin{proof}
The proof requires just a simple verification; the details can be also found in \cite{Nel69}.

\end{proof}

\begin{lemma}\label{lem2.5}
If $T$ commutes with the representation  $\pi_{universal}(\mathfrak A_N)$, then $T$ maps $\L^2(\mu)$ into itself, for every $\mu\in\mathcal M(\K_N)$. 

\end{lemma}

\begin{proof}
Let $x=f\sqrt{d\mu}$ be in $\L^2(\mu)$ so $f\in L^2(\mu)$. Let $Tx=g\sqrt{d\nu}$. We have that $\m_{Tx}\ll\m_x$, by Lemma \ref{lem2.3}. But $d\m_{Tx}=|g|^2\,d\nu$ and $d\m_x=|f|^2\,d\mu$, by Lemma \ref{lem2.3.1}. Therefore $|g|^2\,d\nu\ll\mu $ so, by the Radon-Nikodym theorem there exists $h\geq0$ in $L^1(\mu)$ such that $|g|^2\,d\nu=h\,d\mu$. Then $g\sqrt{d\nu}=\frac{g}{|g|}\sqrt{h}d\mu\in \L^2(\mu)$. 

\end{proof}

We return to the proof of the theorem. If $T$ commutes with $\pi_{universal}(\mathfrak A_N)$ then for every $\mu$, $T$ maps $\L^2(\mu)$ into itself and $T$ commutes with $\pi(f)$ for all bounded measurable functions $f$. Using Lemma \ref{lem2.4}, we pull-pack everything to $L^2(\mu)$ and we obtain an operator that commutes with all multiplication operators and therefore, it must be a multiplication operator too. So there exists a function $F_\mu$ in $L^\infty(\mu)$ such that $T(f\sqrt{d\mu})=F_\mu f\sqrt{d\mu}$, for all $f\in L^2(\mu)$. 

It remains to check the properties of the functions $F_\mu$. We have that $\|F_\mu\|_{L^\infty(\mu)}\leq \|T\|$ and this implies (i). 

If $\mu\ll\lambda$ then, for all $f\in L^2(\mu)$, $f\sqrt{d\mu}=f\sqrt{\frac{d\mu}{d\lambda}}\sqrt{d\lambda}$ so applying $T$ we have $F_\mu f\sqrt{d\mu}=F_\lambda f\sqrt{\frac{d\mu}{d\lambda}}\sqrt{d\lambda}$ which means that 
$F_\mu f\sqrt{\frac{d\mu}{d\lambda}}=F_\lambda f\sqrt{\frac{d\mu}{d\lambda}}$, $\lambda$-a.e. This implies that $F_\mu\sqrt{\frac{d\mu}{d\lambda}}=F_\lambda\sqrt{\frac{d\mu}{d\lambda}}$, $\lambda$-a.e.. This implies further that 
$F_\mu {\frac{d\mu}{d\lambda}}=F_\lambda{\frac{d\mu}{d\lambda}}$, $\lambda$-a.e.. Integrating with respect to $\lambda$, against the characteristic function of any Borel set, we obtain that $F_\mu=F_\lambda$, $\mu$-a.e.. This proves (ii). (iii) is already proved.

For the converse, assume $T$ is given by the functions $F_\mu$ satisfying (i)--(iii). First, we have to check that $T$ is well defined. So take $f\sqrt{d\mu}=g\sqrt{d\nu}$ and let $\lambda$ be a measure such that $\mu,\nu\ll\lambda$. Then $f\sqrt{\frac{d\mu}{d\lambda}}=g\sqrt{\frac{d\nu}{d\lambda}}$, $\lambda$-a.e. 
We have, by (ii), $F_\mu=F_\lambda$, $\mu$-a.e. so $F_\mu{\frac{d\mu}{d\lambda}}=F_\lambda{\frac{d\mu}{d\lambda}}$, $\lambda$-a.e. and therefore $F_\mu\sqrt{\frac{d\mu}{d\lambda}}=F_\lambda\sqrt{\frac{d\mu}{d\lambda}}$, $\lambda$-a.e.. Similarly, $F_\nu\sqrt{\frac{d\nu}{d\lambda}}=F_\lambda\sqrt{\frac{d\nu}{d\lambda}}$, $\lambda$-a.e.. These relations imply that 
$fF_\mu\sqrt{\frac{d\mu}{d\lambda}}=gF_\nu\sqrt{\frac{d\nu}{d\lambda}}$, $\lambda$-a.e., which means that $F_\mu f\sqrt{d\mu}=F_\nu g\sqrt{d\nu}$ and that $T$ is well defined.

(i) implies that $T$ is bounded with $\|T\|\leq\sup_\mu \|F_\mu\|_{L^\infty(\mu)}$. Also, Proposition \ref{pr7.2}, implies that $T$ commutes with $P(A)$ for all Borel subsets $A$ and therefore $T$ commutes with $\pi_{universal}(\mathfrak A_N)$.

An operator $T$ as above commutes with $\pi_{universal}(\O_N)$ iff $T$ commutes with all $S_i$, $i\in\bz_N$ (this follows from the fact that $T$ is normal and by the Fuglede-Putnam theorem it will commute also with $S_i^*$).
This means that 
$TS_i=S_iT$, i.e., for all $f\sqrt{d\mu}$ in $\H(\K_N)$,
$$F_{\mu\circ\sigma_i^{-1}}(f\circ\sigma)\sqrt{d\mu\circ\sigma_i^{-1}}=(F_\mu\circ\sigma)(f\circ\sigma)\sqrt{d\mu\circ\sigma_i^{-1}}.$$
This is equivalent to
$$F_{\mu\circ\sigma_i^{-1}}(f\circ\sigma)=(F_\mu\circ\sigma)(f\circ\sigma),\quad\mu\circ\sigma_i^{-1}\mbox{-a.e.},$$
or
$$F_{\mu\circ\sigma_i^{-1}}=F_\mu\circ\sigma,\quad\mu\circ\sigma_i^{-1}\mbox{-a.e.}.$$
Composing with $\sigma_i$ we get further the equivalence with \eqref{eq2.2.1}.

\end{proof}

\begin{proposition}\label{pr3.6}
Let $(\mu,(f_i)_{i\in\bz_N})$ be a nonnegative monic system. Let $(S_i)_{i\in\bz_N}$ be the associated monic representation of $\O_N$. Then the map $W$ from $L^2(\mu)$ to $\H(\K_N)$ given by $Wf=f\sqrt{d\mu}$, defines an isometric embedding which intertwines the representation with the universal representation of $\O_N$, which we denote here by $(S_i^u)_{i\in\bz_N}$. 

\end{proposition}  
  
  \begin{proof}
  Lemma \ref{lem2.4} shows that $W$ is isometric. We just have to check that it is intertwining. We have, for $f\in L^2(\mu)$, $i\in\bz_N$:
  $$WS_if=W(f_i(f\circ\sigma))=f_i(f\circ\sigma)\sqrt{d\mu}=(f\circ\sigma)\sqrt{|f_i|^2d\mu}$$$$=(f\circ\sigma)\sqrt{d(\mu\circ\sigma_i^{-1})}=S_i^{univ}(f\sqrt{d\mu})=S_i^{univ}Wf.$$
  
  \end{proof}
  \begin{acknowledgements}
This work was partially supported by a grant from the Simons Foundation (\#228539 to Dorin Dutkay). We thank Sergii Bezuglyi  for conversations about ergodic theory. One of the authors has had very helpful conversations with Prof Sergii Bezuglyi about Markov measures.
\end{acknowledgements}

\bibliographystyle{alpha}	
\bibliography{eframes}

\end{document}